\documentclass[12pt,twoside]{amsart}
\usepackage{graphicx}
\usepackage{enumerate}
\usepackage{amsmath, amsthm, amsfonts}

\newtheorem{theorem}{Theorem}[section]

\newtheorem{corollary}[theorem]{Corollary}

\newtheorem{proposition}[theorem]{Proposition}

\newtheorem{definition}[theorem]{Definition}

\begin{document}
\title{\bf On $I_{\alpha^{m}}(A)$, $C_{\alpha^{m}}(A)$ and $\alpha^{m}$-open and closed Functions}
\date{}

\author{R. Parimelazhagan}

\address{Department of Mathematics, RVS Technical Campus, Coimbatore - 641 402, Tamilnadu, India. E-mail: pari$_{_{-}}$kce@yahoo.com.}

\author{Milby Mathew}

\address{Research Scholar, Department of Mathematics, Karpagam University, Coimbatore - 32, Tamilnadu, India. E-mail: milbymanish@yahoo.com}

\maketitle
\begin{abstract}
In this paper, we derive more on $\alpha^{m}$-continuous functions and $\alpha^{m}$-irresolute functions with $\alpha^{m}$-open maps and $\alpha^{m}$-closed maps in topological spaces also we introduce $I_{\alpha^{m}}(A)$ and $C_{\alpha^{m}}(A)$ by using the $\alpha^{m}$-open sets and $\alpha^{m}$-closed sets and studied some of their properties. 
\end{abstract}


\section{Introduction}

Malghan \cite{ma1982} introduced the concept of generalized closed maps in topological spaces. Devi \cite{de1994} introduced and studied $sg$-closed maps and $gs$-closed maps. $wg$-closed maps and $rwg$-closed maps were introduced and studied by Nagavani \cite{na1999}. Regular closed maps, $gpr$-closed maps and $rg$-closed maps have been introduced and studied by Long \cite{lo1978}, Gnanambal \cite{gn1997} and Arockiarani \cite{aroc1997} respectively. \\

 In 2014, \cite{milb2014} we introduced the concepts of $\alpha^{m}$-closed sets and $\alpha^{m}$-open set in topological spaces. Also we have introduced the concepts of $\alpha^{m}$-continuous functiopns and $\alpha^{m}$-irresolute functions. In this paper, by deriving the properties of $\alpha^{m}$-continuous functions and $\alpha^{m}$-irresolute functions with $\alpha^{m}$-open maps and $\alpha^{m}$-closed maps in topological spaces. Further we introduce the concept of $I_{\alpha^{m}}(A)$ and $C_{\alpha^{m}}(A)$ by using the $\alpha^{m}$-open sets and $\alpha^{m}$-closed sets respectively. Further various characterisation are studied. 
are introduced
\section{Preliminaries}
Throughout this paper $(X, \tau)$ and $(Y, \sigma)$ represent topological spaces on which no separation axioms are assumed unless otherwise mentioned. For a subset $A$ of a space $(X, \tau)$, $cl(A)$, $int(A)$ and $A^{c}$ denote the closure of $A$, the interior of $A$ and the complement of $A$ in $X$, respectively. \\

\begin{definition}
A subset $A$ of a topological space $(X, \tau)$ is called \\ \\
$(a)$ a preopen set \cite{ma1982} if $A\subseteq int(cl(A))$ and pre-closed set if $cl(int(A))\subseteq A$. \\ \\
$(b)$ a semiopen set \cite{le1963} if $A \subseteq cl(int(A))$ and semi closed set if $int(cl(A))\subseteq A$. \\ \\
$(c)$ an $\alpha$-open set \cite{nj1965} if $A \subseteq int(cl(int(A)))$ and an $\alpha$-closed set if $cl(int(cl(A)))\subseteq A$. \\ \\
$(d)$ a semi-preopen set \cite{an1986} ($\beta$-open set) if $A \subseteq cl(int(cl(A)))$ and semi-preclosed set if $int(cl(int(A))) \subseteq A$. \\ \\
$(e)$ an $\alpha^{m}$-closed set \cite{milb2014} if $int(cl(A))\subseteq U$, whenever $A\subseteq U$ and $U$ is $\alpha$-open. The complement of $\alpha^{m}$-closed set is called an $\alpha^{m}$-open set. \\
\end{definition}

\begin{definition}
A space $(X,\ \tau_{X})$ is called a $T_{\frac{1}{2}}$-space \cite{le1970} if every $g$-closed set is closed. \\
\end{definition}

\begin{definition} 
A function $f:(X,\ \tau)\longrightarrow(Y,\ \sigma)$ is called \\ \\
$(a)$ an $\alpha^{m}$-continuous \cite{milb} if $f^{-1}(V)$ is $\alpha^{m}$-closed in $(X,\ \tau)$ for every closed set $V$ of $(Y,\ \sigma)$. \\ \\
$(b)$ an $\alpha^{m}$-irresolute \cite{milb} if $f^{-1}(V)$ is $\alpha^{m}$-closed in $(X,\ \tau)$ for every $\alpha^{m}$-closed set $V$ of $(Y,\ \sigma)$. \\
$(b)$ an $\alpha$-irresolute \cite{nj1965} if $f^{-1}(V)$ is $\alpha$-closed in $(X,\ \tau)$ for every $\alpha$-closed set $V$ of $(Y,\ \sigma)$. \\
\end{definition}

\section{More on $\alpha^{m}$-open sets}
\begin{theorem}
A map $f:(X, \tau)\longrightarrow(Y, \sigma)$ is $\alpha^{m}$-closed if and only if for each subset $S$ of $(Y, \sigma)$ and each open set $U$ containing
$f^{-1}(S)$ there is an $\alpha^{m}$-open set $V$ of $(Y, \sigma)$ such that $S\subseteq V$ and $f^{-1}(V)\subseteq U$. \\
\end{theorem}
\begin{proof}
Suppose $f$ is $\alpha^{m}$-closed. Let $S\subseteq Y$ and $U$ be an open set of $(X, \tau)$ such that $f^{-1}(S)\subseteq U$. Then $V=(f(U^{c}))^{c}$ is an $\alpha^{m}$-open set containing $S$ such that $f^{-1}(V)\subseteq U$. \\ \\

For the converse, let $F$ be a closed set of $(X, \tau)$. Then $f^{-1}(f(F^{c}))\subseteq F^{c}$ and $F^{c}$ is open. By assumption, there exists an $\alpha^{m}$-open set $V$ in $(Y, \sigma)$ such that $(f(F^{c}))\subseteq V$ and $f^{-1}(V)\subseteq F^{c}$ and so $F\subseteq (f^{-1}(V))^{c}$. Hence $V^{c}\subseteq f(F)\subseteq f((f^{-1}(V))^{c})\subseteq V^{c}$ which implies $f(F) = V^{c}$. Since $V^{c}$ is $\alpha^{m}$-closed, $f(F)$ is $\alpha^{m}$-closed and therefore $f$ is $\alpha^{m}$-closed. \\
\end{proof}

\begin{proposition}
If $f:(X, \tau)\longrightarrow(Y, \sigma)$ is $\alpha$-irresolute $\alpha^{m}$-closed and $A$ is an $\alpha^{m}$-closed subset of $(X, \tau)$, then $f(A)$ is $\alpha^{m}$-closed in $(Y, \sigma)$. \\
\end{proposition}
\begin{proof}
Let $U$ be an $\alpha$-open set in $(Y, \sigma)$ such that $f(A)\subseteq U$. Since $f$ is $\alpha$-irresolute, $f^{-1}(U)$ is an $\alpha$-open set containing $A$. Hence $cl(A)\subseteq f^{-1}(U)$ as $A$ is $\alpha^{m}$-closed in $(X, \tau)$. Since $f$ is $\alpha^{m}$-closed, $f(cl(A))$ is an $\alpha^{m}$-closed set contained in the $\alpha$-open set $U$, which implies that $cl(f(cl(A)))\subseteq U$ and hence $cl(f(A))\subseteq U$. Therefore, $f(A)$ is an $\alpha^{m}$-closed set in $(Y, \sigma)$. \\
\end{proof}

\begin{corollary}
 3.8. Let $f:(X, \tau)\longrightarrow(Y, \sigma)$ be an $\alpha^{m}$-closed and $g:(Y, \sigma)\longrightarrow(Z, \eta)$ be $\alpha^{m}$-closed and $\alpha$-irresolute, then their composition $g\circ f:(X, \tau)\longrightarrow(Z, \eta)$ is $\alpha^{m}$-closed. \\ 
\end{corollary}
\begin{proof}
Let $A$ be a closed set of $(X, \tau)$. Then by hypothesis $f(A)$ is an $\alpha^{m}$-closed set in $(Y, \sigma)$. Since $g$ is both $\alpha^{m}$-closed and
$\alpha$-irresolute by the above Proposition, $g(f(A)) = (g\circ f)(A)$ is $\alpha^{m}$-closed in $(Z, \eta)$ and therefore $g\circ f$ is $\alpha^{m}$-closed. \\
\end{proof}

\begin{proposition}
Let $f:(X, \tau)\longrightarrow(Y, \sigma)$ be a closed map and $g:(Y, \sigma)\longrightarrow(Z, \eta)$ be an $\alpha^{m}$-closed map, then their composition $g\circ f:(X, \tau)\longrightarrow(Z, \eta)$. \\
\end{proposition}
\begin{proof}
Let $A$ be a closed set of $(X, \tau)$. Then by assumption $f(A)$ is closed in $(Y, \sigma)$ and again by assumption $g(f(A))$ is $\alpha^{m}$-closed in $(Z, \eta)$. i.e., $(g\circ f)(A)$ is $\alpha^{m}$-closed in $(Z, \eta)$ and so $g\circ f$ is $\alpha^{m}$-closed. \\
\end{proof}

\begin{theorem}
Let $f:(X, \tau)\longrightarrow(Y, \sigma)$ and $g:(Y, \sigma)\longrightarrow(Z, \eta)$ be two maps such that their composition $g\circ f:(X, \tau)\longrightarrow(Z, \eta)$ is an $\alpha^{m}$-closed map. Then the following statements are true. \\ \\
$(a)$ If $f$ is continuous and surjective, then $g$ is $\alpha^{m}$-closed. \\ \\
$(b)$ If $g$ is $\alpha^{m}$-irresolute and injective, then $f$ is $\alpha^{m}$-closed. \\ 
\end{theorem}
\begin{proof}
$(a)$ Let $A$ be a closed set of $(Y, \sigma)$. Since $f$ is continuous, $f^{-1}(A)$ is closed in $(X, \tau)$ and since $g\circ f$ is $\alpha^{m}$-closed, $(g\circ f)(f^{-1}(A))$ is $\alpha^{m}$-closed in $(Z, \eta)$. That is $g(A)$ is $\alpha^{m}$-closed in $(Z, \eta)$, since $f$ is surjective. Therefore $g$ is an $\alpha^{m}$-closed map. \\ \\
$(b)$ Let $B$ be a closed set of $(X, \tau)$. Since $g\circ f$ is $\alpha^{m}$-closed, $(g\circ f)(B)$ is $\alpha^{m}$-closed in $(Z, \eta)$. Since $g$ is $\alpha^{m}$-irresolute, $g^{-1}((g\circ f)(B))$ is $\alpha^{m}$-closed set in $(Y, \sigma)$. That is $f(B)$ is $\alpha^{m}$-closed in $(Y, \sigma)$, since $g$ is injective. Thus $f$ is an $\alpha^{m}$-closed map. \\ 
\end{proof}

\begin{definition} $(a)$ Let $X$ be a topological space and let $x\in X$. A subset $N$ of $X$ is said to be $\alpha^{m}$-nbbd of $x$ if there exists an $\alpha^{m}$-open set $G$ such that $x\in G\subset N$.

The collection of all $\alpha^{m}$-nbhd of $x\in X$ is called an $\alpha^{m}$-nbhd system at $x$ and shall be denoted by $n^{\#_{\alpha^{m}}}(x)$. \\ \\
$(b)$ Let $X$ be a topological space and $A$ be a subset of $X$, A subset $N$ of $X$ is said to be $\alpha^{m}$-nbhd of $A$ if there exists an $\alpha^{m}$-open set $G$ such that $A\in G\subseteq N$.\\ \\
$(c)$ Let $A$ be a subset of $X$. A point $x\in A$ is said to be an $\alpha^{m}$-interior point of $A$, if $A$ is an $n^{\#_{\alpha^{m}}}(x)$. The set of all $\alpha^{m}$-interior points of $A$ is called an $\alpha^{m}$-interior of $A$ and is denoted by $I_{\alpha^{m}}(A)$. 

$I_{\alpha^{m}}(A)=\bigcup\{G: G \ is \ \alpha^{m}$-$open,\ G\subset A\}$. \\ \\
$(d)$ Let $A$ be a subset of $X$. A point $x\in A$ is said to be an $\alpha^{m}$-closure of $A$. Then 

$C_{\alpha^{m}}(A)=\bigcap\left\{F:A\subset F\in \alpha^{m}C(X, \tau_{X})\right\}$. \\
\end{definition}

\begin{theorem}
Let $X$ be a topological space. If $A$ is $\alpha^{m}$-closed subset of $X$ and $x\in C_{\alpha^{m}}(A)$ if and only if for any $\alpha^{m}$-nbhd $N$ of $x$ in $X$, $N\cap A\neq\phi$. \\
\end{theorem}
\begin{proof}
Let us assume that there is an $\alpha^{m}$-nbhd $N$ of the point $x$ in $X$ such that $N\cap A=\phi$. There exist an $\alpha^{m}$-open set $G$ of $X$ such that $x\in G\subseteq N$. Therefore we have $G\cap A=\phi$ and so $x\in X-G$. Then $C_{\alpha^{m}}(A)\in X-G$ and therefore $x\notin C_{\alpha^{m}}(A)$, which is the contradiction to the hypothesis $x\in C_{\alpha^{m}}(A)$. Therefore $N\cap A\neq\phi$. \\ \\

Conversely, Suppose that $x\notin C_{\alpha^{m}}(A)$, then there exists a $\alpha^{m}$-closed set $G$ of $X$ such that $A\subseteq G$ and $x\notin G$. Thus $x\in X-G$ and $X-G$ is $\alpha^{m}$-open in $X$ and hence $X-G$ is an $\alpha^{m}$-nbhd of $x$ in $X$. But $A\cap(X-G)=\phi$ which is a contradiction. Hence $x\in C_{\alpha^{m}}(A)$. \\
\end{proof}

\begin{proposition}
If $A$ be a subset of $X$, then $I_{\alpha^{m}}(A)=\bigcup\{G: G \ is \ \alpha^{m}$-$open,\ G\subset A\}$. \\
\end{proposition}
\begin{proof}
Let $A$ be a subset of $X$, $x\in I_{\alpha^{m}}(A)\Longleftrightarrow$ $x$ is an $\alpha^{m}$-interior point of $A$, $A$ is an $n^{\#_{\alpha^{m}}}(x)$ which implies that there exists an $\alpha^{m}$-open set $G$ such that $x\in G\subset A$, $x\in \bigcup\{G: G \ is \ \alpha^{m}$-$open,\ G\subset A\}$. Hence $I_{\alpha^{m}}(A)=\bigcup\{G: G \ is \ \alpha^{m}$-$open,\ G\subset A\}$. \\
\end{proof}

\begin{theorem}
Assume that the collection of all $\alpha^{m}$-open sets of $Y$ is closed under arbitrary union. Let $f:(X, \tau)\longrightarrow(Y, \sigma)$
be a map. Then the following statements are equivalent: \\ \\
$(a)$ $f$ is an $\alpha^{m}$-open map. \\ \\
$(b)$ For a subset $A$ of $(X, \tau)$, $f(int(A))\subseteq I_{\alpha^{m}}(f(A))$. \\ \\
$(c)$ For each $x\in X$ and for each neighborhood $U$ of $x$ in $(X, \tau)$, there exists an $\alpha^{m}$-nbhd $W$ of $f(x)$ in $(Y, \sigma)$ such that
$W\subseteq f(U)$. \\
\end{theorem}
\begin{proof}
$(a)\Longrightarrow (b)$ Suppose $f$ is $\alpha^{m}$-open. Let $A\subseteq X$. Then $int(A)$ is open in $(X, \tau)$ and so $f(int(A))$ is $\alpha^{m}$-open in $(Y, \sigma)$. We have $f(int(A))\subseteq f(A)$. Therefore, $f(int(A))\subseteq I_{\alpha^{m}}(f(A))$. \\ \\
$(b)\Longrightarrow (c)$ Suppose $(b)$ holds. Let $x\in X$ and $U$ be an arbitrary neighborhood of $x$ in $(X, \tau)$. Then there exists an open set $G$ such that $x\in G\subseteq U$. By assumption, $f(G) = f(int(G))\subseteq I_{\alpha^{m}}(f(G))$. This implies $f(G) = I_{\alpha^{m}}(f(G))$, we have $f(G)$ is $\alpha^{m}$-open in $(Y, \sigma)$. Further, $f(x)\in f(G)\subseteq f(U)$ and so $(c)$ holds, by taking $W = f(G)$. \\ \\
$(c)\Longrightarrow (a)$ Suppose $(c)$ holds. Let $U$ be any open set in $(X, \tau)$), $x\in U$ and $f(x) = y$. Then $y\in f(U)$ and for each $y\in f(U)$, by assumption there exists an $\alpha^{m}$-nbhd $W_{y}$ of $y$ in $(Y, \sigma)$ such that $W_{y}\subseteq f(U)$. Since $W_{y}$ is an $\alpha^{m}$-nbhd of $y$, there exists an $\alpha^{m}$-open set $V_{y}$ in $(Y, \sigma)$ such that $y\in V_{y}\subseteq W_{y}$. Therefore, $f(U)=\cup\{V_{y}: y\in f(U)\}$ is an $\alpha^{m}$-open set in $(Y, \sigma)$ by the given condition. Thus $f$ is an $\alpha^{m}$-open map. \\
\end{proof}

\begin{corollary}
A map $f:(X, \tau)\longrightarrow(Y, \sigma)$ is $\alpha^{m}$-open if and only if $f^{-1}(C_{\alpha^{m}}(B))\subseteq cl(f^{-1}(B))$ for each subset $B$ of $(Y, \sigma)$. \\
\end{corollary}
\begin{proof}
Suppose that $f$ is $\alpha^{m}$-open. Then for any $B\subseteq Y$, $f^{-1}(B)\subseteq cl(f^{-1}(B))$, there exists an $\alpha^{m}$-closed set
$K$ of $(Y, \sigma)$ such that $B\subseteq K$ and $f^{-1}(K)\subseteq cl(f^{-1}(B))$. Therefore, $f^{-1}(C_{\alpha^{m}}(B))\subseteq (f^{-1}(K))\subseteq cl(f^{-1}(B))$, since $K$ is an $\alpha^{m}$-closed set in $(Y, \sigma)$. \\ \\

Conversely, let $S$ be any subset of $(Y, \sigma)$ and $F$ be any closed set containing $f^{-1}(S)$. Put $K = C_{\alpha^{m}}(S)$. Then $K$ is an $\alpha^{m}$-closed set and $S\subseteq K$. By assumption, $f^{-1}(K) = f^{-1}(C_{\alpha^{m}}(S))\subseteq cl(f^{-1}(S))\subseteq F$ and therefore, $f$ is $\alpha^{m}$-open. \\
\end{proof}


\end{document}